\documentclass[a4paper, reqno, 11pt]{amsart}

\usepackage[english]{babel}
\usepackage{amsmath}
\usepackage{mathrsfs}
\usepackage{amssymb}
\usepackage{amsthm}
\usepackage{enumerate}
\usepackage{ifthen}
\usepackage{bbm}
\usepackage{color}
\usepackage{xcolor}
\usepackage{graphicx}
\provideboolean{shownotes} 
\setboolean{shownotes}{true}
\usepackage{hyperref}
\usepackage{geometry}
\usepackage{comment}

\newcommand{\margnote}[1]{
\ifthenelse{\boolean{shownotes}}%
{\marginpar{\raggedright\tiny\texttt{#1}}}%
{}%
}

\newcommand{\hole}[1]{
\ifthenelse{\boolean{shownotes}}%
{\begin{center} \fbox{ \rule {.25cm}{0cm}
\rule[-.1cm]{0cm}{.4cm} \parbox{.85\textwidth}{\begin{center}
\texttt{#1}\end{center}} \rule {.25cm}{0cm}}\end{center}}
{}
}
\newtheorem{thm}{Theorem}[section]

\newtheorem{prop}[thm]{Proposition}

\newtheorem{rem}[thm]{Remark}

\newtheorem{defn}[thm]{Definition}
\theoremstyle{definition}

\newcommand{\e}{\varepsilon}		       
\newcommand{\R}{\mathbb{R}}
\newcommand{\T}{\mathbb{T}}

\newcommand{\Z}{\mathbb{Z}}
\newcommand{\dive}{\mathop{\mathrm {div}}}

\newcommand{\de}{\,\mathrm{d}}

\newcommand{\uv}{u^{\nu}}

\newcommand{\omv}{\omega^{\nu}}

\numberwithin{equation}{section}

\subjclass[2020]{35Q30, 35Q31, 35Q35, 76B03}
\keywords{Incompressible Euler equations; vorticity; propagation of regularity; modulus of continuity; inviscid limit.}

\begin{document}

\title[Propagation of regularity and inviscid limit for 2D Euler]{Propagation of logarithmic regularity and inviscid limit for the 2D Euler equations}

\author[G. Ciampa]{Gennaro Ciampa}
\address[G.\ Ciampa]{DISIM - Dipartimento di Ingegneria e Scienze dell'Informazione e Matematica\\ Universit\`a  degli Studi dell'Aquila \\Via Vetoio \\ 67100 L'Aquila \\ Italy}
\email[]{\href{gciampa@}{gennaro.ciampa@univaq.it}}
\author[G. Crippa]{Gianluca Crippa}
\address[G. Crippa]{Departement Mathematik Und Informatik\\ Universit\"at Basel \\Spiegelgasse 1 \\CH-4051 Basel \\ Switzerland}
\email[]{\href{gianluca.crippa@}{gianluca.crippa@unibas.ch}}
\author[S. Spirito]{Stefano Spirito}
\address[S. Spirito]{DISIM - Dipartimento di Ingegneria e Scienze dell'Informazione e Matematica\\ Universit\`a  degli Studi dell'Aquila \\Via Vetoio \\ 67100 L'Aquila \\ Italy}
\email[]{\href{stefano.spirito@}{stefano.spirito@univaq.it}}

\begin{abstract}
The aim of this note is to study the Cauchy problem for the 2D Euler equations under very low regularity assumptions on the initial datum. We prove propagation of regularity of logarithmic order in the class of weak solutions with $L^p$ initial vorticity, provided that $p\geq 4$. We also study the inviscid limit from the 2D Navier-Stokes equations for vorticity with logarithmic regularity in the Yudovich class, showing a rate of convergence of order $|\log\nu|^{-\alpha/2}$ with $\alpha>0$.
\end{abstract}
\maketitle
\begin{center}
{\em “Dedicated to Pierangelo Marcati on the occasion of his 70th birthday"}
\end{center}
\vspace{0,5cm}
\section{Introduction}
We consider the Cauchy problem for the two-dimensional incompressible Euler equations in vorticity formulation
\begin{equation}\label{eq:vort}
\begin{cases}
\partial_t\omega+u\cdot\nabla\omega =0,\\
u(t,\cdot)=\nabla^\perp(-\Delta)^{-1}\omega(t,\cdot),\\
\omega(0,\cdot)=\omega_0,
\end{cases}
\end{equation}
where the unknowns are the vorticity $\omega:[0,T]\times\T^2\to \R$ and the velocity $u:[0,T]\times\T^2\to \R^2$, while $\omega_0:\T^2\to \R$ is a given initial datum with zero average, i.e.
$$
\int_{\T^2}\omega_0(x)\de x=0.
$$
The coupling between the velocity and the vorticity in the second line of \eqref{eq:vort} is known as the {\em Biot-Savart law}, and in particular it implies the incompressibility condition $\dive u=0$. The understanding of the well-posedness of these equations represents one of the classic problems of mathematical fluid dynamics; we refer to \cite{MB}  for an overview of the available theory.
In the two-dimensional case the Euler equations in vorticity formulation are a non-linear and non-local transport equation driven by an incompressible velocity field. Thus, a simple energy estimate for smooth solutions gives that
$$
\|\omega(t,\cdot)\|_{L^p}=\|\omega_0\|_{L^p} \qquad \forall t>0,
$$
for all $1\leq p\leq \infty$. It is therefore natural to study weak solutions of \eqref{eq:vort} in an $L^p$ framework.  We highlight from the outset that we work on the two-dimensional torus $\T^2$ to reduce technical details, but all the results can be adapted to the whole-space case $\R^2$ requiring $\omega_0\in L^1\cap L^p(\R^2)$. The existence of weak solutions with $\omega_0\in L^p$ with $p>1$ has been proved by DiPerna and Majda in \cite{DPM}, see also \cite{De} for positive measures Radon measures $\omega_0\in H^{-1}_{\rm loc}$ and \cite{VW} for $\omega_0\in L^1$. Uniqueness of weak solutions is known only in the case of bounded vorticity $\omega\in L^\infty$ (see \cite{Y}) or slightly less (see \cite{CrippaStefani, Y2}). However, the uniqueness of weak solutions in $L^p$ with $p<\infty$ is still an open and very challenging problem. Recently several non-uniqueness results have been proved (see \cite{BM, BS, BrC, MeS, VI, VII}) but the complete picture is far from being fully understood.\\

In these notes we investigate two problems related to \eqref{eq:vort}. The first one is the {\em propagation of regularity} for weak solutions of \eqref{eq:vort}. Roughly speaking, it means the following: given a Banach space $Y$ we want to understand whether the information that the initial datum $\omega_0$ belongs to $Y$ implies that the solution $\omega(t,\cdot)\in Y$ for every $t>0$. In \cite{Beirao} it is shown that if one considers Dini continuous initial data $\omega_0$, the 2D Euler equations \eqref{eq:vort} admits a unique global solution within these critical spaces. In particular, the propagation of the Dini semi-norm provides an $L^\infty$ bound on $\nabla u$ with a constant that grows exponentially in time, see also \cite{Koch}.
In the case the velocity field is not Lipschitz there is no propagation in general due to the lack of Lipschitz regularity of the flow. However, one can have a control on the loss of regularity. This is indeed the case for Yudovich's solutions for which the velocity is only log-Lipschitz.
In this regards, in \cite{BC} it has been shown that if $\omega_0\in L^\infty \cap W^{\alpha,p}$ then, for all $0<\beta<\alpha$, the unique bounded solution $\omega$ of \eqref{eq:vort} belongs to the space $W^{\beta(t),p}$ where
\begin{equation}
\beta(t):=\beta\exp\left(-\int_0^t V(\tau)\de \tau\right),\qquad
V(t):=\sup_{0<|x-x'|\leq 1}\frac{|u(t,x)-u(t,x')|}{|x-x'|(1-\log|x-x'|)}.
\end{equation}
We point out that the quantity $\int_0^t V(\tau)\de \tau$ is finite since $u$ is log-Lipschitz.  
Later on, in \cite{Cozzi} it is showed that if $\omega_0\in W^{\alpha,p}$ (with $\alpha p\leq 2$, $p>1$ and $\alpha\in(0,2)$) is a {\em continuous} function, then the unique bounded solution $\omega(t,\cdot)\in W^{\beta,p}$ for any $t\in[0,T]$ and any $0<\beta<\alpha$.
These results have been improved recently in \cite{BN MA}, where the authors showed that solutions in the Yudovich class $\omega\in L^\infty([0,T]\times\T^2)\cap C([0,T];L^1(\T^2))$ satisfy the following: let $0<\alpha\leq 1$ and $p\geq 1$, we have that
\begin{itemize}
\item[(i)] If $\omega_0\in W^{\alpha,p}(\T^2)$ then
$$
[\omega(t,\cdot)]_{W^{\alpha(t),p}}\lesssim_{\alpha,p}\|\omega_0\|_{L^\infty}+[\omega_0]_{W^{\alpha,p}},
$$
for any $t\in[0, T]$, where $\alpha(t)=\frac{\alpha}{1+C\|\omega_0\|_{L^\infty}\alpha pt}$.
\item[(ii)] If $\omega_0\in C(\T^2)\cap W^{\alpha,p}(\T^2)$ with $p>1$ then $\omega(t,\cdot)\in W^{\beta,p}(\T^2)$ for any $0<\beta<\alpha$ and any $t\in[0, T]$. When $\alpha=1$ we also have $\omega(t,\cdot)\in W^{1,p'}(\T^2)$ for any $1\leq p'<p$ and any $t\in[0, T]$.
\item[(iii)] If $\omega_0\in W^{\alpha,p}(\T^2)$ with $p>2/\alpha$ then $\omega(t,\cdot)\in W^{\alpha,p}(\T^2)$ for any $t\in[0, T]$.
\end{itemize}
In a similar fashion, in \cite{Jeong} the author constructed an example of initial datum $\omega_0\in L^\infty\cap W^{1,p}$ with $p>2$ such that the unique bounded solution $\omega$ continuously loses {\em integrability} in time, i.e. $\omega(t,\cdot)$ belongs to the Sobolev space $W^{1,p(t)}$ with $p(t)$ being a decreasing function of time. On the other hand, log-Hölder coefficients of Yudovich's solution of \eqref{eq:vort} are conserved if one assumes that the velocity is Lipschitz, see \cite{Chae-Jeong}.\\
For more irregular initial data, in \cite{CDE} the authors showed that if $\omega_0\in L^\infty\cap B^s_{p,\infty}$, for some $s>0$ and $p\geq 1$, the unique solution $\omega$ of \eqref{eq:vort} satisfies $\omega(t,\cdot)\in L^\infty\cap B^{s(t)}_{p,\infty}$ with $s(t):=s \exp(-Ct\|\omega_0\|_\infty)$. Here $B^s_{p,\infty}$ denotes the usual Besov space. Moreover, they suggest that the loss may be improved to a polynomial law as in \cite{BN MA}, see \cite[Remark 3]{CDE}. 

Our purpose is to look for a (possibly more general) class of initial data so that the regularity is propagated without any loss. The motivation comes from recent results on the propagation of regularity for the linear transport equation driven by irregular velocity fields, i.e. the equation
\begin{equation}\label{eq:te}
\begin{cases}
\partial_t\theta+b\cdot\nabla\theta=0,\\
\theta(0,\cdot)=\theta_0.
\end{cases}
\end{equation}
It is well known that if the vector field is less regular than Lipschitz, the regularity of the initial data can be immediately lost, see \cite{ACM}. We refer to \cite{BCD} for a classical reference in which the case of Besov regularity with Log-Lipschitz field is analyzed.
On the other hand, in \cite{BN} it has been shown that bounded solutions of \eqref{eq:te} propagate {\em regularity of logarithmic order} provided that $b$ is a divergence-free vector field in $L^1([0,T];W^{1,p}(\T^d))$. In detail, for any $\alpha>0$ we define $H^{\log,\alpha}$ to be the functional space
$$
H^{\log,\alpha}(\T^d):=\left\{f\in L^2(\T^d): [f]_{H^{\log,\alpha}}^2:=\int_{B_{1/3}}\int_{\T^d}\frac{|f(x+h)-f(x)|^2}{|h|^d}\frac{\de x\de h}{\log(1/|h|)^{1-\alpha}}<\infty\right\},
$$
which is a Banach space endowed with the norm $\|f\|_{H^{\log,\alpha}}^2:=\|f\|_{L^2}+[f]_{H^{\log,\alpha}}^2$.
The authors of \cite{BN} show the following bound on the $[\,\cdot\,]_{H^{\log,\alpha}}$ semi-norm
\begin{equation}\label{stima prop log}
[\theta(t,\cdot)]_{H^{\log,p}}\lesssim_{p,d}\left(\int_0^t\|\nabla b(s,\cdot)\|_{L^p}\de s\right)^\frac{p}{2}\|\theta_0\|_{L^\infty}+[\theta_0]_{H^{\log,p}},
\end{equation}
with $\alpha=p>1$. The same result has been reproduced in \cite{Meyer-Seis} using an equivalent Besov-type semi-norm and Littlewood–Paley's theory.
Going back to the Euler equations, it is immediate to check that the bound \eqref{stima prop log} holds for solutions in the Yudovich class (as pointed out in Theorem \ref{teo:propagazione yudovich} below). However, our aim is to go beyond the class of bounded solutions. To do this, we must consider a slightly bigger functional space but still with the same scaling of $H^{\log,\frac12}$, namely the space $W^2_{\log,\frac12}$ introduced in \cite{BJ, BJ2}: we will say that a function $f\in L^2(\T^d)$ belongs to $W^2_{\log,\frac12}(\T^d)$ if 
$$
\sup_{0<h\leq 1/2}\frac{1}{|\log h|^{\frac12}}\int_{\T^{d}}\int_{\T^{d}}K_h(x-y)|u(x)-u(y)|^2\de x\de y<\infty,
$$
where the kernel $K_h$ is a positive, bounded, smooth and symmetric function defined as
$$
K_h(x)=\frac{1}{(|x|+h)^d},\qquad \mbox{ for }|x|<1/2,
$$
see Section 3 for the precise definition. Then, we can prove that any solution of \eqref{eq:vort} arising from initial data $\omega_0\in L^p\cap W^2_{\log,\frac12}$ with $p\geq 4$ satisfies $\omega(t,\cdot)\in L^p\cap W^2_{\log,\frac12}$. The result is the following.
\begin{thm}\label{thm:main1}
Let $\omega_0\in L^4\cap W^2_{\log,\frac12}(\T^2)$ and let $\omega\in L^\infty([0,T]; L^4(\T^2))$ be any weak solution of \eqref{eq:vort} starting from $\omega_0$. Then, there exists a constant $C>0$ such that $\omega$ satisfies
\begin{equation}
[\omega(t,\cdot)]_{W^2_{\log,\frac12}}\leq [\omega_0]_{W^2_{\log,\frac12}}+C\sqrt{t}\|\omega_0\|_{L^2}^\frac12\|\omega_0\|_{L^4},\qquad \mbox{for any }t\in [0,T].
\end{equation}
\end{thm}
Remarkably, Theorem \ref{thm:main1} holds in a class in which uniqueness is not known, but all the solutions in $L^\infty([0,T];L^4(\T^2))$ are renormalized in the sense of DiPerna-Lions \cite{DPL}, as shown in \cite{FMN}. Moreover, in addition to improving the result in terms of integrability of the initial datum $\omega_0$, the proof is substantially different from the one in \cite{BN} being more Eulerian in spirit. In particular, we do not use the Lagrangian representation of the solutions (which still holds as a consequence of the renormalization property) and we use a commutator estimate proved in \cite[Proposition 13]{BJ2}. We observe here {\em en passant} that the Littlewood-Paley approach of \cite{Meyer-Seis} could also work under the hypothesis of Theorem \ref{thm:main1}, but we preferred not to follow this route in order to make the proof as simple as possible.

Finally, we believe that Theorem \ref{thm:main1} may be generalized to a larger class of interpolation spaces which ultimately includes $H^{\log,\alpha}$ and, more in general, logarithmic Besov spaces $B^{\mathrm{log},b}_{2,q}$. A comprehensive approach to these spaces may be found in \cite{DT}. We plan to address this issue in a forthcoming paper.
\\

Our second interest is the inviscid limit from the 2D incompressible Navier-Stokes equations
\begin{equation}\label{eq:vort-ns-intro}\tag{NS}
\begin{cases}
\partial_t\omega^\nu+u^\nu\cdot\nabla\omega^\nu =\nu\Delta\omega^\nu,\\
u^\nu(t,\cdot)=\nabla^\perp(-\Delta)^{-1}\omega^\nu(t,\cdot),\\
\omega^\nu(0,\cdot)=\omega_0^\nu,
\end{cases}
\end{equation}
where $\nu> 0$ is the kinematic viscosity of the fluid, and $\omega_0^\nu$ is a zero-average initial datum (possibly depending on $\nu$) such that
$$
\omega_0^\nu\to\omega_0, \qquad\mbox{in }L^p(\T^2).
$$ 
We recall that in two dimensions solutions of the Navier-Stokes equations \eqref{eq:vort-ns-intro} are regular if the initial datum is square integrable, i.e. for $\omega_0^\nu\in L^2(\T^2)$. Here we look for rates of convergence as the viscosity $\nu\to 0$ when the spatial domain is the torus $\T^2$, so that no boundary layers have to be considered. In this regards, the case of smooth initial data $u_0\in H^s$ (with $s>2$) was analyzed by Masmoudi in \cite{Ma} who showed that
\begin{equation}
\|u^\nu(t,\cdot)-u(t,\cdot)\|_{H^{s'}}\lesssim
\begin{cases}
\nu t,\qquad &\mbox{ if }s'\leq s-2,\\
(\nu t)^{(s-s')/2}, &\mbox{ if }s-2\leq s'\leq s-1,
\end{cases}
\end{equation}
together with the implications on the vorticity side ($\omega_0\in H^s$ with $s>1$). We also mention the recent work \cite{LucaRickVV} in which the authors found a rate of convergence of order $\nu$ {\em uniform in time} for time-quasi-periodic solutions of \eqref{eq:vort} with a small time-quasi-periodic external force. \\
For Yudovich's solutions, the following rate of convergence for the velocity field was proved by Chemin in \cite{Ch}
\begin{equation}
\sup_{t\in[0,T]}\|u^\nu(t,\cdot)-u(t,\cdot)\|_{L^2}\leq (4\nu T)^{\frac12 \exp(-CT\|\omega_0\|_{L^\infty})}\|\omega_0\|_{L^2\cap L^\infty}e^{1-\exp(-CT\|\omega_0\|_{L^\infty})},\label{rate chemin}
\end{equation}
for some constant $C>0$. We also refer to \cite{S} for a $\log$ improvement of this rate.
Concerning rates of convergence for the vorticity, a similar power law rate has been showed in \cite{CDE} under the additional assumption $\omega_0\in B^{s}_{p,\infty}$ with $s>0$ and $p\geq 1$. In particular, in \cite{CDE} the authors prove the following rate 
\begin{equation}\label{rate cde}
\|\omega^\nu(t,\cdot)-\omega(t,\cdot)\|_{L^p}\lesssim (\nu t)^{\frac{s(t)}{1+s(t)}},\qquad
s(t):=\exp(-Ct\|\omega_0\|_\infty).
\end{equation}
It is important to point out that this rate has been obtained by combining losing estimates for the  Besov regularity \big(i.e. $\omega^\nu(t,\cdot)\in B^{s(t)}_{p,\infty}$\big) together with the rate for the velocity \eqref{rate chemin}. Notice that \eqref{rate chemin} and \eqref{rate cde} are not uniform in time and they actually deteriorate exponentially fast. In this regards, the improvement of the losing estimate proved in \cite{BN MA} implies that the rate \eqref{rate cde} holds with an exponent that deteriorates polynomially in time, i.e. $\tilde{s}(t)=\frac{s}{1+Ctsp}$, as pointed out in \cite[Remark 3]{CDE}.
In \cite{CCS4} we proved a related result for Yudovich's solutions but with a different approach: without requiring any kind of regularity on the initial data, we showed that it is possible to obtain the rate of convergence
\begin{equation}\label{rate nostro}
\|\omega^\nu(t,\cdot)-\omega(t,\cdot)\|_{L^p}\lesssim \max\{\phi(\nu),\nu^\beta\},
\end{equation}
where $\beta$ depends on $\|\omega_0\|_\infty,T$, and $p$, while $\phi:\R^+\to\R^+$ is a continuous function with $\phi(0)=0$.
In detail, the function $\phi$ is a modulus of continuity depending on the initial datum $\omega_0$ such that
$$
\|\omega_0(\cdot+h)-\omega_0\|_{L^p}\leq \phi(|h|),\qquad \mbox{for $|h|$ small enough}.
$$
Notice that the regularity of the initial datum is somehow encoded in the function $\phi$ and it can be made quantitative assuming some regularity on $\omega_0$.
Our purpose here is somehow between \cite{CCS4} and \cite{CDE}: we aim to prove a rate which does not deteriorate in time requiring some regularity on the initial datum. To this end, we look for a regularity class for the initial datum that is rougher than Besov and which provides an explicit form for the function $\phi$ appearing in \eqref{rate nostro}. Once again, the inviscid limit for advection-diffusion equations with Sobolev velocity fields (see \cite{BCC, BN, Meyer-Seis}) and the propagation of logarithmic regularity suggest to look for a logarithmic rate of convergence. 

Our second main result is the following. 
\begin{thm}\label{thm:main2}
Let $\omega_0\in L^\infty\cap H^{\log,\alpha}(\T^2)$ for some $\alpha>0$. Let $\omega$ and $\omega^\nu$ be, respectively, the unique bounded solutions of the Euler and Navier-Stokes equations arising from $\omega_0$. Then, there exists a constant $C>0$ depending on $\alpha,T,\|\omega_0\|_{H^{\log,\alpha}}$, and $\|\omega_0\|_{L^\infty}$ such that
\begin{equation}\label{rate log}
\sup_{t\in(0,T)}\|\omega^\nu(t,\cdot)-\omega(t,\cdot)\|_{L^2}\leq \frac{C}{|\log \nu|^{\alpha/2}}.
\end{equation}
\end{thm}
The strategy of the proof relies on the stochastic Lagrangian representation of solutions of the Navier-Stokes equations as in \cite{CCS4}, i.e. the solution $\omega^\nu$ has an explicit formula in terms of the stochastic flow of $u^\nu$ which is given by
$$
\omv(t,x)=\mathbb{E}[\omv_0(X_{t,0}^\nu(x))],
$$
with $\mathbb{E}$ denoting the expectation, and $X_{t,s}^\nu$ is the solution of the stochastic differential equation
$$
\begin{cases}
\de X^\nu_{t,s}(x,\xi)=u^\nu(s,X^\nu_{t,s}(x,\xi))\de s + \sqrt{2\nu} \de W_s(\xi),\hspace{0.5cm}s\in[0,t),\\
X^\nu_{t,t}(x,\xi)=x.
\end{cases}
$$
The rate \eqref{rate nostro} is then obtained exploiting the estimate on the difference quotients of functions belonging to $H^{\log,\alpha}$ (see Theorem \ref{teo:diff-quot} below) and the convergence in the zero-noise limit of the stochastic flow $X^\nu$ towards the deterministic flow of the limit solution. We remark once again that our result provides a rate of convergence for a class of initial data more irregular than the one considered in \cite{CDE}. Moreover, it is worth noting that the constant $C$ in \eqref{rate log} diverges for $T\to\infty$, but the rate of convergence is independent of $T$.

\section{Functional setting}
In this section we fix the notations and we recall some preliminary results.
We start by recalling the definitions and some properties of the spaces of functions with derivatives of logarithmic order from \cite{BN,BN20}. 
For any $\alpha>0$ we define the space
$$
H^{\log,\alpha}(\T^d):=\left\{f\in L^2(\T^d): \int_{B_{1/3}}\int_{\T^d}\frac{|f(x+h)-f(x)|^2}{|h|^d}\frac{\de x\de h}{\log(1/|h|)^{1-\alpha}}<\infty\right\},
$$
and the corresponding semi-norm
\begin{equation}\label{def:log-seminorm}
[f]_{H^{\log,\alpha}}^2:=\int_{B_{1/3}}\int_{\T^d}\frac{|f(x+h)-f(x)|^2}{|h|^d}\frac{\de x\de h}{\log(1/|h|)^{1-\alpha}}.
\end{equation}
The space $H^{\log,\alpha}(\T^d)$ is a Banach space endowed with the norm
\begin{equation}
\|f\|_{H^{\log,\alpha}}^2:=\|f\|_{L^2}^2+[f]_{H^{\log,\alpha}}^2.
\end{equation}
Using the Fourier representation, the following characterization is shown to hold in \cite{BN}:
\begin{equation}\label{fourier Hlog}
\|f\|_{H^{\log,\alpha}}^2\sim_{\alpha,d}\sum_{k\in\Z^d}\log(2+|k|)^\alpha |\hat{u}(k)|^2.
\end{equation}
We now give a precise statement for the theorem on the propagation of regularity to which we referred in the introduction, see \cite[Corollary 1.2]{BN}.
\begin{thm}\label{cor:propagation transport}
Let $b\in L^1([0,T];W^{1,p}(\T^d))$ be a divergence-free vector field for some $p>1$. Then, any solution $\theta\in L^\infty([0,T]\times\T^d)$ of \eqref{eq:te} satisfies
\begin{equation}
[\theta(t,\cdot)]_{H^{\log,p}}\lesssim_{p,d}\left(\int_0^t\|\nabla b(s,\cdot)\|_{L^p}\de s\right)^\frac{p}{2}\|\theta_0\|_{L^\infty}+[\theta_0]_{H^{\log,p}},
\end{equation}
for any $t\in[0,T]$.
\end{thm}

For any $f\in H^{\log,\alpha}(\T^d)$ we define for any $x\in\T^d$ the function
\begin{equation}
L_\alpha f(x):=\left(\int_{B_{1/3}}\frac{|f(x+h)-f(x)|^2}{|h|^d}\frac{\de h}{\log(1/|h|)^{1-\alpha}}\right)^{\frac12},
\end{equation}
and it follows that $L_\alpha f\in L^2(\T^d)$ with 
\begin{equation}\label{norma=seminorma}
\|L_\alpha f\|_{L^2}=[f]_{H^{\log,\alpha}}. 
\end{equation}
The following estimate on the difference quotients holds, see \cite[Theorem 1.11]{BN20}.
\begin{thm}\label{teo:diff-quot} 
Let $\alpha>0$ be fixed. For any $f\in H^{\log,\alpha}(\T^d)$ it holds
\begin{equation}
|f(x)-f(y)|\leq C(d,\alpha)\log(1/(|x-y|))^{-\alpha/2}\left(L_{\alpha}f(x)+L_{\alpha}f(y)\right),
\end{equation}
for any $x,y\in \T^d$ such that $|x-y|<1/36$.
\end{thm}

We point out that the previous theorem is proved in \cite{BN20} in functional spaces that are a generalization of $H^{\log,\alpha}$ in a non-Hilbertian framework. Precisely, one can consider the spaces $X^{\gamma,p}$ which are defined as follows.
\begin{defn}
Let $p\in(0,\infty)$ and $\gamma\in(0,\infty)$ be fixed. We define 
\begin{equation}\label{def:log gamma-seminorm}
[f]_{X^{\gamma,p}(\T^d)}:=\left(\int_{B_{1/3}}\int_{\T^d}\frac{|f(x+h)-f(x)|^p}{|h|^d}\frac{\de x\de h}{\log(1/|h|)^{1-p\gamma}}\right)^{1/p},
\end{equation}
and we set
\begin{equation}
X^{\gamma,p}(\T^d):=\{f\in L^p(\T^d):[f]_{X^{\gamma,p}(\T^d)}<\infty\}.
\end{equation}
\end{defn}
For $p\geq 1$, the space $X^{\gamma,p}(\T^d)$ is a Banach space endowed with the norm
\begin{equation}
\|f\|_{X^{\gamma,p}}=\|f\|_{L^p}+[f]_{X^{\gamma,p}},
\end{equation}
and the Gagliardo semi-norm $[f]_{X^{\gamma,p}}$ is lower semicontinuous with respect to the strong topology of $L^p$.
With these notations, we have that $H^{\log,\alpha}=X^{\frac{\alpha}{2},2}$. Moreover, the semi-norms $[\,\cdot\,]_{H^{\log,\alpha}}$ are increasing in $\alpha$, i.e for any $0<\alpha\leq \alpha'<\infty$ it holds
$$
[f]_{H^{\log,\alpha}}\leq [f]_{H^{\log,\alpha'}},
$$
see \cite[Proposition 1.3]{BN20}.

\section{Propagation of regularity}
The aim of this section is to analyze the propagation of logarithmic regularity for weak solutions of the 2D Euler equations. For solutions in the Yudovich class the propagation follows immediately by Theorem \ref{cor:propagation transport} as shown in the following. 
\begin{thm}\label{teo:propagazione yudovich}
Let $\omega_0\in L^\infty\cap H^{\log,p}(\T^2)$ for some $p> 1$. Then, the unique solution $\omega\in L^\infty([0,T]\times\T^2)$ of \eqref{eq:vort} satisfies
\begin{equation}
[\omega(t,\cdot)]_{H^{\log,p}}\lesssim_p t^\frac{p}{2}\|\omega_0\|_{L^\infty}^{1+\frac{p}{2}}+[\omega_0]_{H^{\log,p}},\qquad \mbox{ for any }\,\,0\leq t\leq T.
\end{equation}
\end{thm}
\begin{proof} 
By Theorem \ref{cor:propagation transport}, for any $0\leq t\leq T$ the unique bounded solution $\omega$ satisfies
\begin{equation}\label{stima 1}
[\omega(t,\cdot)]_{H^{\log,p}}\lesssim_{p}\left(\int_0^t\|\nabla u(s,\cdot)\|_{L^p}\de s\right)^\frac{p}{2}\|\omega_0\|_{L^\infty}+[\omega_0]_{H^{\log,p}}.
\end{equation}
By the properties of the Biot-Savart operator, we have that
\begin{equation}\label{stima biot savart}
\|\nabla u(s,\cdot)\|_{L^p}\lesssim_p \|\omega(s,\cdot)\|_{L^p}\leq \|\omega_0\|_{L^\infty}, \qquad \mbox{ for any }\,\,0\leq s\leq T.
\end{equation}
where in the last inequality we used that the spatial domain is $\T^2$ together with the basic energy estimate on $\omega$.
Substituting \eqref{stima biot savart} in \eqref{stima 1} the proof is completed.
\end{proof}

We now consider the case of unbounded vorticity. To this end, we need to introduce some preliminary definitions and results from \cite{BJ, BJ2, BrJ} that we adapt to our context.
\begin{defn}
Let $0<\theta<1$ and define the semi-norms
\begin{equation}\label{seminorme jabin}
[u]_{\theta}^2:=\sup_{0<h\leq 1/2}\frac{1}{|\log h|^{\theta}}\int_{\T^{d}}\int_{\T^{d}}K_h(x-y)|u(x)-u(y)|^2\de x\de y,
\end{equation}
where the kernel $K_h$ is a positive, bounded and symmetric function defined as
\begin{equation}
K_h(x)=\frac{1}{(|x|+h)^d},\qquad \mbox{ for }|x|<1/2,
\end{equation}
independent of $h$ for $|x|\geq 2/3$, equal to a positive constant outside $B(0,3/4)$, and periodized so as to belong in $C^\infty(\T^d\setminus B(0,3/4))$.
\end{defn}
Notice that \eqref{seminorme jabin} defines a semi-norm because it vanishes if $u$ is a constant. Moreover, the semi-norms are decreasing in $\theta$, i.e.
\begin{equation}\label{monotonia seminorme}
[u]_{\theta}\leq [u]_{\theta'},\,\,\mbox{if }\theta'\leq \theta.
\end{equation}
Correspondingly, we define the space $W^2_{\log,\theta}$ as follows
\begin{equation}
W^2_{\log,\theta}:=\{u\in L^2(\T^d):[u]_{\theta}<\infty\},
\end{equation}
which is a Banach space endowed with the norm
\begin{equation}
\|u\|_{2,\theta}^2=\|u\|_{L^2}^2+[u]_{\theta}^2.
\end{equation}
The following proposition holds, see \cite[Proposition 1]{BJ2}.
\begin{prop}\label{prop:verso 1}
For any $s>0$ and any $0<\theta<1$ one has the (compact) embeddings
$$
W^{s,2}(\T^d)\subset W^2_{\log,\theta}(\T^d)\subset L^2(\T^d).
$$
In addition, using the Fourier representation it holds
\begin{equation}
\|u\|_{L^2}^2+[u]_{\theta}^2\sim \sup_h\sum_{k\in\Z^d\setminus\{0\}}\frac{1+\left|\log\left(\frac{1}{|k|}+h\right)\right|}{|\log h|^\theta}|\hat{u}(k)|^2\leq \sum_{k\in\Z^d}\log(1+|k|)^{1-\theta}|\hat{u}(k)|^2.\label{norme equivalenti upper}
\end{equation}
\end{prop}
In view of the Proposition \ref{prop:verso 1} and the equivalence in \eqref{fourier Hlog}, for $0<\theta<1$ we have the inclusion
\begin{equation}\label{eq:inclusione}
H^{\log,1-\theta}(\T^d)\subset W^{2}_{\log,\theta}(\T^d).
\end{equation}
We remark that the inclusion in \eqref{eq:inclusione} is strict.
We now recall the following commutator estimate from \cite{BJ2} and adapted to our context, see \cite[Proposition 13]{BJ2}.
\begin{prop}
\label{prop:commutatore bj}
Let $a\in W^{1,p}(\T^d)$ be a divergence-free vector field with $1\leq p\leq 2$. Then, there exists a constant $C>0$ depending only on $p$ and $d$ such that for all $g\in L^{p'}(\T^d)$ with $\frac{1}{p}+\frac{1}{p'}=1$,
\begin{equation}\label{stima che mi serve}
\int_{\T^{d}}\int_{\T^{d}}\nabla K_h(x-y)(a(x)-a(y))|g(x)-g(y)|^2\de x\de y\leq C|\log h|^{\frac12}\|\nabla a\|_{L^p}\|g\|_{L^{2p'}}^2.
\end{equation}
\end{prop}

We can finally prove our first main result.
\begin{thm}\label{teo:propagazione Lp}
Let $\omega_0\in L^4\cap W^2_{\log,\frac12}(\T^2)$ and let $\omega\in L^\infty([0,T]; L^4(\T^2))$ be any weak solution of \eqref{eq:vort} starting from $\omega_0$. Then, there exists a constant $C>0$ such that $\omega$ satisfies the following bound
\begin{equation}
[\omega(t,\cdot)]_{W^2_{\log,\frac12}}\leq [\omega_0]_{W^2_{\log,\frac12}}+C\sqrt{t}\|\omega_0\|_{L^2}^\frac12\|\omega_0\|_{L^4},\qquad \mbox{for any }t\in [0,T].
\end{equation}
\end{thm}
\begin{rem}
The constant in the statement of Theorem \eqref{teo:propagazione Lp} depends depends on the constant in \eqref{stima che mi serve} of Proposition \ref{prop:commutatore bj} with $d=2$ and $p=p'=2$.
\end{rem}
\begin{proof}
Let $\omega\in L^\infty([0,T]; L^4(\T^2))$ be a weak solution of $\eqref{eq:vort}$ and let $\rho^\e$ be a family of smooth mollifiers. Then, the function $\omega^\e=\omega*\rho^\e$ satisfies the equation
\begin{equation}\label{equazione mollificata}
\begin{cases}
\partial_t \omega^\e+u\cdot\nabla\omega^\e=r^\e,\\
u(t,x)=K*\omega(t,\cdot)(x),\\
\omega^\e(0,\cdot)=\omega_0*\rho^\e,
\end{cases}
\end{equation}
where the commutator $r^\e$ is defined as
\begin{equation}
r^\e:=u\cdot\nabla (\omega*\rho^\e)-\left(u\cdot\nabla \omega\right)*\rho^\e.
\end{equation}
In particular, since $\omega,\nabla u\in L^\infty([0,T];L^4(\T^2))$ we have that $r^\e\to 0$ in $L^1([0,T];L^2(\T^2))$, see \cite[Lemma II.1]{DPL}.
Thus, by using the equation \eqref{equazione mollificata}, the function $|\omega^\e(t,x)-\omega^\e(t,y)|^2$ satisfies the equation
\begin{align}
\partial_t |\omega^\e(t,x)-\omega^\e(t,y)|^2+[u(t,x)\cdot\nabla_x+u(t,y)\cdot\nabla_y]|\omega^\e(t,x)-\omega^\e(t,y)|^2\nonumber\\
=2(r^\e(t,x)+r^\e(t,y))(\omega^\e(t,x)-\omega^\e(t,y)).\label{equazione differenza}
\end{align}
Then, we use \eqref{equazione differenza} and $\dive u=0$ to compute
\begin{align*}
\frac{\de}{\de t}\int_{\T^{2}}\int_{\T^{2}} & K_h(x-y)|\omega^\e(t,x)-\omega^\e(t,y)|^2\de x\de y\\
&=\int_{\T^{2}}\int_{\T^{2}} \nabla K_h(x-y)(u(t,x)-u(t,y))|\omega^\e(t,x)-\omega^\e(t,y)|^2\de x\de y\\
&+2\int_{\T^{2}}\int_{\T^{2}} K_h(x-y)(r^\e(t,x)+r^\e(t,y))(\omega^\e(t,x)-\omega^\e(t,y))\de x\de y.
\end{align*}
We integrate in time and we use Proposition \ref{prop:commutatore bj} to obtain that 
\begin{align*}
\int_{\T^{2}}\int_{\T^{2}} & K_h(x-y)|\omega^\e(t,x)-\omega^\e(t,y)|^2\de x\de y\leq \int_{\T^{2}}\int_{\T^{2}} K_h(x-y)|\omega^\e_0(x)-\omega^\e_0(y)|^2\de x\de y\\
&+2\int_0^T\int_{\T^{2}}\int_{\T^{2}} K_h(x-y)(r^\e(t,x)+r^\e(t,y))(\omega^\e(t,x)-\omega^\e(t,y))\de x\de y\de t\\
&+C t |\log h|^{\frac12}\|\nabla u\|_{L^\infty L^2}\|\omega^\e\|_{L^\infty L^4}^2,
\end{align*}
for all $0\leq t\leq T$. Since any weak solution $\omega\in L^\infty([0,T]; L^p(\T^2))$ with $p\geq 2$ is renormalized (see \cite{FMN}), the following bounds hold
\begin{align*}
\|\omega^\e\|_{L^\infty L^4}&\leq \|\omega_0\|_{L^4},\\
\|\nabla u\|_{L^\infty L^2}&\leq C \|\omega\|_{L^\infty L^2}\leq C\|\omega_0\|_{L^2},
\end{align*}
and then we get that 
\begin{align*}
\int_{\T^{2}}\int_{\T^{2}} & K_h(x-y)|\omega^\e(t,x)-\omega^\e(t,y)|^2\de x\de y\leq \int_{\T^{2}}\int_{\T^{2}} K_h(x-y)|\omega^\e_0(x)-\omega^\e_0(y)|^2\de x\de y\\
&+2\int_0^T\int_{\T^{2}}\int_{\T^{2}} K_h(x-y)(r^\e(t,x)+r^\e(t,y))(\omega^\e(t,x)-\omega^\e(t,y))\de x\de y\de t\\
&+C t |\log h|^{\frac12}\|\omega_0\|_{L^2}\|\omega_0\|_{L^4}^2.
\end{align*}
Finally, the convergence of the commutator $r^\e$ allows us to take the limit $\e\to 0$ obtaining that
\begin{align*}
\int_{\T^{2}}\int_{\T^{2}} K_h(x-y)|&\omega(t,x)-\omega(t,y)|^2\de x\de y\\
&\leq \int_{\T^{2}}\int_{\T^{2}} K_h(x-y)|\omega_0(x)-\omega_0(y)|^2\de x\de y+C t |\log h|^{\frac12}\|\omega_0\|_{L^2}\|\omega_0\|_{L^4}^2,
\end{align*}
and the definition of the semi-norm $[\,\cdot\,]_{\frac12}$ in \eqref{seminorme jabin} implies that
\begin{equation}
[\omega(t,\cdot)]_{\frac12}^2\leq [\omega_0]_{\frac12}^2+C t\|\omega_0\|_{L^2}\|\omega_0\|_{L^4}^2.
\end{equation}
The proof is complete.
\end{proof}

\section{Vanishing viscosity limit}
In this last section we study the inviscid limit of the 2D Navier-Stokes equations \eqref{eq:vort-ns-intro}. Our goal is to provide a logarithmic rate of convergence (in the viscosity) assuming that the initial vorticity belongs to the space $H^{\log,\alpha}$.\\

We start by introducing the Stochastic Langrangian representation of \eqref{eq:vort-ns-intro}. Let $(\Omega, \mathcal{F}, \mathbb{P})$ be a given probability space, we define the map $X^\nu:[0,T]\times[0,T]\times \T^2\times\Omega\to\T^2$ as follows. For $\mathbb{P}$-a.e. $\xi\in\Omega$ and for any $t\in(0,T)$ and any $s\in[0,T]$ we consider a $\T^2$-valued Brownian motion $W_s$ adapted to the backward filtration, i.e. satisfying $W_t=0$. The map $s\mapsto X^{\nu}_{t,s}(x,\xi)$ is obtained by solving 
\begin{align}
&\begin{cases}
\de X^\nu_{t,s}(x,\xi)=u^\nu(s,X^\nu_{t,s}(x,\xi))\de s + \sqrt{2\nu} \de W_s(\xi),\hspace{0.5cm}s\in[0,t),\\
X^\nu_{t,t}(x,\xi)=x.
\end{cases}\label{eq:nsl1}
\end{align}
For $\mathbb{P}$-a.e. $\xi\in\Omega$ the map $x\in\T^2\mapsto X^{\nu}_{t,s}(x,\xi)\in\T^2$ is measure-preserving for any $t\in[0,T]$ and $s\in[0,t]$ (see \cite{LBL}). Moreover, by the Feynman-Kac formula (see \cite{K, LBL}), the function
\begin{equation}
\omv(t,x)=\mathbb{E}[\omv_0(X_{t,0}^\nu(x))]
\end{equation}
satisfies the advection-diffusion equation
\begin{equation*}
\partial_{t}\omv+\uv\cdot\nabla\omv-\nu\Delta\omv=0,
\end{equation*}
with initial datum $\omv_0$, where we have denoted by $\mathbb{E}[f]$ the expectation, i.e. the average with respect to $\mathbb{P}$. As usual, we will omit the explicit dependence on the parameter $\xi\in\Omega$. Therefore, the couple
\begin{align}
&\,\,\uv(t,x):=\nabla^\perp(-\Delta)^{-1}\omv(t,\cdot)(x),\label{eq:nsl2}\\
&\,\,\omv(t,x):=\mathbb{E}[\omega^{\nu}_0(X^{\nu}_{t,0}(x))]\label{eq:nsl3},
\end{align}
solves the Cauchy problem for the Navier-Stokes equations \eqref{eq:vort-ns-intro}. The couple $(\uv,\omv)$ defined by the equations \eqref{eq:nsl2} and \eqref{eq:nsl3} is the {\em Lagrangian representation} of solutions to \eqref{eq:vort-ns-intro}.\par
 
We remark that the probability space and the Brownian motion can be arbitrarily chosen. Thus, the Lagrangian representation does not depend on the probability space. Indeed, since $\uv$ is a smooth function, the equation \eqref{eq:nsl1} is satisfied in the strong sense \cite{K, LBL}, namely one can find a solution $X^{\nu}_{t,\cdot}$ to \eqref{eq:nsl1} on any given filtered probability space with any given adapted Brownian motions as described above.\\

Finally, we recall the following theorem proved in \cite[Theorem 2.8]{CCS4}.

\begin{thm}\label{convergenza flussi}
Let $\omega_0\in L^\infty(\T^2)$ with $\|\omega_0\|_{L^\infty}=M$. Let $(u,\omega)$ and $(u^\nu,\omega^\nu)$ be, respectively, the unique bounded solutions of the Euler and Navier-Stokes equations with the same initial datum $\omega_0$. Denote with $X$ and $X^\nu$ the corresponding deterministic and stochastic flows. Then, for any $T>0$ there exists a constant $\beta(M,T)$ such that
\begin{equation}\label{rate flusso CCS4}
\sup_{s,t\in[0,T]} \mathbb{E}\left[\int_{\T^2}|X^\nu_{t,s}(x)-X_{t,s}(x)|^2\de x\right]\leq C\nu^{\beta(M,T)}.
\end{equation}
\end{thm}

We can now prove our second main result, which we rewrite for the reader's convenience.
\begin{thm}
Let $\omega_0\in L^\infty\cap H^{\log,\alpha}(\T^2)$ for some $\alpha>0$. Let $\omega$ and $\omega^\nu$ be, respectively, the unique bounded solutions of the Euler and Navier-Stokes equations arising from $\omega_0$. Then, there exists a constant $C>0$ depending on $\alpha,T,\|\omega_0\|_{H^{\log,\alpha}}$, and $\|\omega_0\|_{L^\infty}$ such that
\begin{equation}
\sup_{t\in(0,T)}\|\omega^\nu(t,\cdot)-\omega(t,\cdot)\|_{L^2}\leq \frac{C}{|\log \nu|^{\alpha/2}}.
\end{equation}
\end{thm}
\begin{proof}
Let $\e>0$ be a parameter that we will fix later. We use the Feynman-Kac formula to write
\begin{align*}
\|\omega^\nu(t,\cdot)-\omega(t,\cdot)\|_{L^2}^2&=\int_{\T^2}|\omega^\nu(t,x)-\omega(t,x)|^2\de x \\
&=\int_{\T^2}|\mathbb{E}[\omega_0(X^\nu_{t,0})]-\omega_0(X_{t,0})|^2\de x\\
&\leq \iint_{\{|X^\nu_{t,0}-X_{t,0}|\leq\e\}}|\omega_0(X^\nu_{t,0})-\omega_0(X_{t,0})|^2\de \mathbb{P}\de x\\
&+\iint_{\{|X^\nu_{t,0}-X_{t,0}|>\e\}}|\omega_0(X^\nu_{t,0})-\omega_0(X_{t,0})|^2 \de \mathbb{P}\de x\\
&:=I+II.
\end{align*}
We start by estimating $I$: if we assume that $\e<1/36$, we apply Theorem \ref{teo:diff-quot} and we have that
\begin{align*}
\iint_{\{|X^\nu_{t,0}-X_{t,0}|\leq\e\}}|\omega_0(X^\nu_{t,0})-\omega_0(X_{t,0})|^2\de \mathbb{P}\de x&\leq \frac{C(\alpha)}{|\log\e|^{\alpha}}\mathbb{E}\left[\int_{\T^2}\left[L_{\alpha}\,\omega_0(X^\nu_{t,0})^2+L_{\alpha}\,\omega_0(X_{t,0})^2\right]\de x\right]\\
&\leq\frac{C(\alpha)}{|\log \e|^{\alpha}}[\omega_0]_{H^{\log,\alpha}}^2, 
\end{align*}
where in the last line we used \eqref{norma=seminorma} and the measure preserving property of $X^\nu$ and $X$. To estimate $II$ we use the fact that $\omega_0$ is bounded, Chebishev’s inequality and the convergence of the flows in Theorem \ref{convergenza flussi} to obtain that
\begin{align*}
\iint_{\{|X^\nu_{t,0}-X_{t,0}|>\e\}}|\omega_0(X^\nu_{t,0})-\omega_0(X_{t,0})|^2 \de \mathbb{P}\de x&\leq \frac{C\|\omega_0\|_{L^\infty}^2}{\e^2}\mathbb{E}\left[\int_{\T^2}|X^\nu_{t,0}(x)-X_{t,0}(x)|^2\de x\right]\\
&\leq \frac{C\|\omega_0\|_{L^\infty}^2}{\e^2}\nu^{\beta(M,T)}.
\end{align*}
Thus, by defining $\e:=\nu^{\beta(M,T)/4}$, we finally get 
\begin{align*}
\|\omega^\nu(t,\cdot)-\omega(t,\cdot)\|_{L^2}^2&\leq \frac{C(\alpha)[\omega_0]_{H^{\log,\alpha}}^2}{|\log \e|^{\alpha}}+C\|\omega_0\|_{L^\infty}^2\nu^{\beta(M,T)/2}\\
&=\frac{C(\alpha)\alpha(M,T)^{-\alpha}[\omega_0]_{H^{\log,\alpha}}^2}{|\log \nu|^{\alpha}}+C\|\omega_0\|_{L^\infty}^2\nu^{\beta(M,T)/2}\\
&\leq \frac{C(\alpha,T,\|\omega_0\|_{H^{\log,\alpha}},\|\omega\|_{L^\infty})}{|\log \nu|^{\alpha}},
\end{align*}
where in the last line we used that the logarithm converges slower than any power. This concludes the proof.
\end{proof}

\begin{rem}
An easy interpolation argument implies the convergence of $\omega^\nu$ towards $\omega$ in all $L^q$ spaces with $1\leq q<\infty$ and rate
\begin{equation}
\sup_{t\in(0,T)}\|\omega^\nu(t,\cdot)-\omega(t,\cdot)\|_{L^q}\leq \frac{C}{|\log \nu|^{f(\alpha,q)}},\qquad \mbox{with }f(\alpha,q):=\min\left\{\frac{\alpha}{2}, \frac{\alpha}{q} \right\}.
\end{equation}
\end{rem}

\subsection*{Acknowledgments} During the preparation of this manuscript, G. Ciampa has been supported by the ERC STARTING GRANT 2021 ``Hamiltonian Dynamics, Normal Forms and Water Waves" (HamDyWWa), Project Number: 101039762. G. Crippa is supported by the Swiss National Science Foundation through the project 212573 FLUTURA (Fluids, Turbulence, Advection) and by the SPP 2410 ``Hyperbolic Balance Laws in Fluid Mechanics: Complexity, Scales, Randomness (CoScaRa)'' funded by the Deutsche Forschungsgemeinschaft (DFG, German Research Foundation) through the project 200021E\_217527 funded by the Swiss National Science Foundation. The work of G. Ciampa and S. Spirito is partially supported by INdAM-GNAMPA, by the projects PRIN 2020 ``Nonlinear evolution PDEs, fluid dynamics and transport equations: theoretical foundations and applications”, PRIN2022
``Classical equations of compressible fluids mechanics: existence and properties of non-classical solutions'', and PRIN2022-PNRR ``Some mathematical approaches to climate change and its impacts.'' Views and opinions expressed are however those of the authors only and do not necessarily reflect those of the European Union or the European Research Council. Neither the European Union nor the granting authority can be held responsible for them.

The authors are grateful to Rapha\"el Danchin and to \'Oscar Dom\'inguez  for helpful discussions on the topics of this work.

\end{document}